\documentclass[12pt, reqno]{amsart}

\usepackage[utf8]{inputenc}
\usepackage[T1]{fontenc}
\usepackage{tikz,tikz-3dplot}
\usepackage{amsmath,amsfonts,amssymb,amsmath,latexsym,mathrsfs}
\usepackage{shuffle}
\usepackage[all]{xy}
\usepackage{stackengine}
 
\usepackage{enumerate}

\usepackage{hyperref}
 
\usepackage{tabu}
\usepackage{tikz-cd} 

\usepackage{comment}
\usepackage{url}
\usepackage{tikz}

\usepackage{xfrac}

\newtheorem{theorem}{Theorem}[section]

\newtheorem{proposition}[theorem]{Proposition}

\newtheorem{corollary}[theorem]{Corollary}

\newtheorem{lemma}[theorem]{Lemma}

\begin{document}

\title{Isoclinic groups and conjugacy quandles}

\author{Mohamad \textsc{Maassarani} }
\maketitle
\begin{abstract} We show that two finite isoclinic groups of the same order have isomorphic conjugacy quandles. We also show that the converse hold under some assumptions on the groups. It happens that isoclinsim of group of the same order is equivalent to having isomorphic quandles for groups of order $n<128$. We construct two non isoclinic groups of order $128$ having isomorphic conjugacy quandles. \end{abstract}
 \section*{Introduction and main results}
Isoclinsim is relation on groups introduced by P.Hall. It is known that two isoclinc groups of the same order has same conjugacy class size statistics, i.e. the number of conjugacy classes of a given size of both groups are equal. In these notes we prove a stronger fact that is two isoclinc groups of the same order have isomorphic conjugacy quandles, meaning there is a bijection $\psi$ between both groups such that $\psi(ghg^{-1})=\psi(g)\psi(h)\psi(g)^{-1}$.\\\\
The proof of the last fact is given in section $1$, where we also provide an explicit quandle isomorphism between the conjugacy quandle of the generalized quaternion group on $4n$ elements and the conjugacy quandle of the dihedral group on $4n$ elements.\\\\
In section $2$, we are intrested in the converse of the implication : isoclinc of same order implies isomorphic conjugacy quandles. We show that the converse holds under assumptions on the groups as the assumptions that both groups : are centerless, both centers of the groups intersect the derived groups trivially or both groups has trivial Bogomolov multiplier.\\\\
The last section, section $3$, contains results obtained using theoretical arguments and the software GAP 4.15.1. We exploit that the Bogomolov multiplier of groups of order $n<128$ is almost always trivial to show that two groups of order $n<128$ are isoclinc if and only if they have isomorphic conjugacy quandles. We use the existence of a given symmetric $2$-cocycle over a group of order $64$ proven in \cite{MMb} using GAP to construct two groups of order $128$ that have isomorphic conjugacy quandles but are not isoclinic. 
\section{Isoclinc groups of the same order has isomorphic conjugacy quandles}
A non empty set equipped with a binary operation $Q\times Q \to Q, (x,y)\mapsto x\triangleright y$ is a \textit{quandle} if it satisfies the following three axioms :
\begin{itemize}
\item[1)] For $x\in Q$, $x\triangleright x = x$.
\item[2)] For $x,z \in Q$, there exist a $y\in Q$ such that $x\triangleright y=z$.
\item[3)] For $x,y,z\in Q$, $x\triangleright (y\triangleright z)=(x\triangleright y)\triangleright (x\triangleright z)$.
\end{itemize}
A map $f:Q\to Q'$ of quandles is called a \textit{quandle morphism} if for $x,y\in Q$ $f(x\triangleright y)=f(x)\triangleright f(y)$. If a quandle morphism is bigective it is a \textit{quandle isomorphism}. A group $G$ equipped with the binary operation $x\triangleright y=xyx^{-1}$ is a quandle called the \textit{conjugacy quandle} of $G$ and is usually denoted by $Conj(G)$. \\\\
Recall that the commutator map $[\:\:,\:\:]_G :G\times G \to G, (g,g')\mapsto [g,g']_G=gg'g^{-1}g'^{-1}$ factors throught a commutator map $G/Z(G) \times G/Z(G)\to G$, where $Z(G)$ is the center of $G$. for $g \in G$ we will denote by $\bar{g}$ the class of $g$ in $G/Z(G)$ and use also the notation $[\:\:, \:\:]_G$ for the induced commuator $G/Z(G) \times G/Z(G)\to G$. We hence have $[\bar{g},\bar{g}']_G=[g,g']_G$. Two groups $G$ and $H$ are \textit{isoclinc} if there exist group isomorphisms $\alpha :G/Z(G)\to H/Z(H)$ and $\phi : G' \to H'$ ($K'$ denotes the derived subgroup of $K$) such that the following diagram commutes :
$$\begin{tikzcd}
G/Z(G)\times G/Z(G)\arrow{r}{\alpha \times \alpha}\arrow{d}{[\:\:,\:\:]_G} &H/Z(H)\times H/Z(H) \arrow{d}{[\:\:,\:\:]_H} \\ 
G' \arrow{r}{\phi}  & H' 
\end{tikzcd} $$

\begin{proposition}
Let $G$ be a finite group and let $s_G$ be a section of the morphism $G\to G/Z(G)^{ab}$ $($$ab$ is for the abelianisation$)$.
\item[1)] There exist $g_1,\dots,g_{n_G}$ in the center of $G$ with $n_G=\frac{\vert Z(G) \vert}{\vert Z(G) \cap G' \vert}$ such that the distinct elements of $Z(G)G'$ are the elements $g_ig$ for $i\in\{1,\dots,n_G\}$ and $g\in G'$.
\item[2)] The distinct elements of $G$ are the elements $g_igs_G(x)$ for $i\in \{1,\dots,n_G\}$, $g\in G'$ and $x\in G/Z(G)^{ab}$.
\end{proposition}
\begin{proof}
$1)$ follows from the fact that $Z(G)G'/G'$ is isomorphic to $Z(G)/Z(G)\cap G'$ and $2)$ is a consequence of $1)$ and the fact that $G/Z(G)^{ab}$ is isomorphic $G/Z(G)G'$.
\end{proof}

\begin{lemma}
Let $G$ and $H$ be isoclinc groups having the same order and  $\alpha :G/Z(G)\to H/Z(H)$ and $\phi : G' \to H'$ be the associated isomorphisms.
\item[1)] We have $\frac{\vert Z(G) \vert}{\vert Z(G) \cap G' \vert}=\frac{\vert Z(H) \vert}{\vert Z(H) \cap H' \vert}$.
\item[2)] The isomorphism $\alpha :G/Z(G)\to H/Z(H)$ induces an isomorphism $\bar{\alpha} : G/Z(G)^{ab} \to H/Z(H)^{ab}$.
\item[3)] If $s_G$ is a section to $G \to  G/Z(G)^{ab}$ then $H\to  H/Z(H)^{ab}$ admits a section $s_H$ such that $\alpha(\overline{s_G(x)})=\overline{s_H(\bar{\alpha}(x))}$.
\item[4)] For all $g\in G'$, we have $\alpha(\bar{g})=\overline{\phi(g)}$.
\end{lemma}
\begin{proof}
We prove $1)$. Since $G/Z(G)$ and $H/Z(H)$ are isomoprhic and $G$ and $H$ have the same order we have that $\vert Z(G) \vert=\vert Z(H) \vert $. We need to prove that $\vert Z(G)\cap G'\vert =\vert Z(H) \cap H' \vert$. Since $G/Z(G)$ and $H/Z(H)$ are isomorphic there derived subgroups are isomorphic. But there derived subgroup are isomorphic to $G'/Z(G)\cap G'$ and $H'/Z(G)\cap H'$. Since the group are isoclinic $\vert G' \vert =\vert H'\vert$ and hence $\vert Z(G)\cap G'\vert =\vert Z(H) \cap H' \vert$. This completes the proof of $1)$. $2)$ is clear. We prove $3)$. The map $f:G/Z(G)^{ab}\to G/Z(G),x\to \overline{s_G(x)}$ is a section to $G/Z(G)\to G/Z(G)^{ab}$. Take $f'=f\circ \bar{\alpha}^{-1}$. We have that $\alpha(\overline{s_G(x)})=f'(\bar{\alpha}(x))$. Hence by composing $f'$ by any section to $H\to H/Z(H)$ we get a section $s_H$ as in the lemma. We will prove $4)$. Both sides are multiplicative with respect to product of commutators. Hence, we only need to prove the identity for commutators. Take $g=[g_1,g_2]_G$. We have that  : 
\begin{equation*}
\begin{split}
\alpha(\overline{[g_1,g_2]_G})&=\alpha([\bar{g}_1,\bar{g}_2]_{G/Z(G)}) \\
&=[\alpha(\bar{g}_1),\alpha(\bar{g}_2)]_{H/Z(H)} \\
&=\overline{[\alpha(\bar{g}_1),\alpha(\bar{g}_2)]}_H\\
&=\overline{\phi([\bar{g}_1,\bar{g}_2]_G)}\\
&=\overline{\phi([g_1,g_2]_G)}
\end{split}
\end{equation*}
This proves $4)$ for commutators and hence proves $4)$ since as we have seen that both sides of $4)$ are multiplicative with respect to product of commutators.
\end{proof}

\begin{theorem}
If $G$ and $H$ are two isoclinc finite groups of the same order then their conjugacy quandles $Conj(G)$ and $Conj(H)$ are isomorphic.
\end{theorem}
\begin{proof}
Let $G$ and $H$ be as in the theorem. We hence have two compatible isomorphisms $\alpha :G/Z(G)\to H/Z(H)$ and $\phi : G' \to H'$ 
andthe map $\alpha$ induces an isomorphism on the abelianisation $\bar{\alpha} : G/Z(G)^{ab} \to H/Z(H)^{ab}$. Let $s_G$ be a section to the morphism $G\to G/Z(G)^{ab}$ and $s_H$ be a section to $H\to H/Z(H)^{ab}$ as in $3)$ of the previous lemma. By $1)$ of the previous lemma and the previous proposition the distinct elements of $G$ and $H$ are the elements :
$g_igs_G(x)$ and $h_ihs_H(y)$ for $i\in\{1,\dots,n\}$ $g\in G'$, $h\in H'$, $ x\in G/Z(G)^{ab}$ and $y\in H/Z(H)^{ab}$ and where $n=n_G=n_H$ and $g_1,\dots,g_n $ are some elements of $Z(G)$ and $h_1,\dots,h_n$ are some elements in $Z(H)$. We hence have a bijection $\psi : G\to H$ biven by :
$$\psi(g_igs_G(x))=h_i\phi(g)s_H(\bar{\alpha}(x)),$$
for $i\in \{1,\dots,n\}$, $g\in G'$ and $x\in G/Z(G)^{ab}$. We will prove that $\psi$ is a quandle morphism, i.e $\psi(aba^{-1})=\psi(a)\psi(b) \psi(a)^{-1}$ for $a,b\in G$. Take two elements $g_igs_G(x)$ and $g_jg's_G(x')$ of $G$ as above. We have that : 
\begin{equation*}
\begin{split}
\psi(g_igs_G(x) g_jg's_G(x')(g_igs_G(x))^{-1})&=\psi([g_igs_G(x),g_jg's_G(x')]_Gg_jg's_G(x'))\\
&= \psi(g_j([g_igs_G(x),g_jg's_G(x')]_Gg')s_G(x'))\\
&=h_j\phi([g_igs_G(x),g_jg's_G(x')]_Gg')s_H(\bar{\alpha}(x'))\\
&=h_j\phi([gs_G(x),g's_G(x')]_Gg')s_H(\bar{\alpha}(x'))\\
&=h_j\phi([gs_G(x),g's_G(x')]_G)\phi(g')s_H(\bar{\alpha}(x'))
\end{split}
\end{equation*}
On the other hand :
\begin{equation*}
\begin{split}
\psi(g_igs_G(x))\psi(g_jg's_G(x'))\psi(g_igs_G(x))^{-1}&=h_i\phi(g)s_H(\bar{\alpha}(x))h_j\phi(g')s_H(\bar{\alpha}(x'))(h_i\phi(g)s_H(\bar{\alpha}(x)))^{-1} \\
&=[h_i\phi(g)s_H(\bar{\alpha}(x)),h_j\phi(g')s_H(\bar{\alpha}(x'))]_Hh_j\phi(g')s_H(\bar{\alpha}(x'))\\
&=h_j[\phi(g)s_H(\bar{\alpha}(x)),\phi(g')s_H(\bar{\alpha}(x'))]_H\phi(g')s_H(\bar{\alpha}(x'))
\end{split}
\end{equation*}
Compairing the results of the two last equations we get that $\psi$ is a quandle morphism if and only if :
\begin{equation}\label{eq1}
\phi([gs_G(x),g's_G(x')]_G)=[\phi(g)s_H(\bar{\alpha}(x)),\phi(g')s_H(\bar{\alpha}(x'))]_H
\end{equation}
Using the compatibility of $\phi$ and $\alpha$, the fact that we have chosen $s_H$ such that $\alpha(\overline{s_G(x)})=\overline{s_H(\bar{\alpha}(x))}$ for all $x\in G/Z(G)^{ab}$, then applying $4)$ of the last lemma we get :
\begin{equation}
\begin{split}
\phi([gs_G(x),g's_G(x')]_G)&=[\alpha(\overline{gs_G(x)}),\alpha(\overline{g's_G(x')})]_H \\
&= [\alpha(\overline{g}) \alpha(\overline{s_G(x)}),\alpha(\overline{g'}) \alpha(\overline{s_G(x')})]_H \\
&=[\alpha(\overline{g}) \overline{s_H(\bar{\alpha}(x)}),\alpha(\overline{g'})\overline{s_H(\bar{\alpha}(x'))}]_H \\
&=[\overline{\phi(g)} \: \overline{s_H(\bar{\alpha}(x)}),\overline{\phi(g')}\:\overline{s_H(\bar{\alpha}(x'))}]_H\\
&=[\overline{\phi(g)s_H(\bar{\alpha}(x)}),\overline{\phi(g')s_H(\bar{\alpha}(x'))}]_H\\
&=[\phi(g)s_H(\bar{\alpha}(x)),\phi(g')s_H(\bar{\alpha}(x'))]_H
\end{split}
\end{equation}
We have proved that equation (\ref{eq1}) holds and hence $\psi$ is a quandle morphism. This proves the theorem since we have seen that $\psi$ is a bijection.
\end{proof} 

We will denote by $D_{4n}$ and $Q_{4n}$ the dihedral and the generalized quaternion group on $4n$ elements. $Q_{4n}$ and $D_{4n}$ are isoclinc. It follows from the theorem that they have isomorphic conjugacy quandles. We will give an explicit isomorphism between $Conj(D_{4n})$ and $Conj(Q_{4n})$.
Recall that the dihedral group $D_{4n}$ admits the following presentation :
\[ D_{4n} = \langle r, f \mid r^{2n} = s^2 = 1, \; srs = r^{-1} \rangle \]
and that $Q_{4n}$ admits the following presentation : 
\[ Q_{4n} = \langle a, b \mid a^{2n} = 1, \; b^2 = a^n, \; bab^{-1} = a^{-1} \rangle \]
The $4n$ distinct elements of $D_{4n}$ are : $$1,r,r^2,\dots,r^{2n-1} \quad \text{and} \quad s,rs,r^2s,\dots r^{2n-1}s.$$
The $4n$ distinct elements of $D_{4n}$ are : $$1,a,a^2,\dots,a^{2n-1} \quad \text{and} \quad b,ab,a^2b,\dots a^{2n-1}b.$$
Let $\phi :D_{4n} \to Q_{4n}$ be the map defined by :
$$\phi(r^i)=a^i \quad \text{and} \quad \phi(r^is)=a^ib,$$
for $i=0,\dots,2n-1$. It follows from the enumeration of the elements of $D_{4n}$ and $Q_{4n}$ that the map $\phi$ is bijective. 
\begin{proposition}
The map $\phi :Conj(D_{4n}) \to Conj(Q_{4n})$ is a bijective quandle morphism.
\end{proposition}
\begin{proof}We have seen that $\phi$ is bijective.  We will prove that $\phi$ is a quandle morphism. Take $i,j \in \{ 0,\dots,2n-1\}$. We have :
$$\phi(r^i\triangleright r^j)=\phi(r^j)=a^j=a^i\triangleright a^j=\phi(r^i)\triangleright \phi(r^j).$$
We also have :
$$\phi(r^i\triangleright r^js)=\phi(r^{2i+j}s)=a^{\overline{2i+j}}b=a^{2i+j}b=a^i\triangleright a^jb=\phi(r^i)\triangleright \phi(r^js),$$
where $\overline{2i+j}$ is the remainder of $2i+j$ modulo $2n$.
One also has that : 
$$\phi(r^is\triangleright r^j)=\phi(r^{-j})=a^{-j}=a^ib\triangleright a^j=\phi(r^is)\triangleright \phi(r^j).$$
Finaly :
$$\phi(r^is\triangleright r^js)=\phi(r^{2i-j}s)=a^{2i-j}b=a^ib\triangleright a^jb=\phi(r^is)\triangleright \phi(r^js).$$
This completes the proof of the fact that $\phi$ is a quandle morphism.
\end{proof}
 \section{Study of the converse}
For $G$ a group and $g\in G$ we will denote by $c_g\in Inn(G)$ the conjugation by $g$.
\begin{proposition}
Let $\psi : G \to H$ be a quandle isomorphism.
\item[1)] $\psi$ restricts to a bijection between $Z(G) $ and $Z(H)$.
\item[2)] We have a well defined isomorphism $Inn(\psi) : Inn(G) \to Inn(H)$ given by $$Inn(\psi)(c_g)=c_{\psi(g)}.$$
\end{proposition}
\begin{proof}
An element $z\in Z(G)$ satisfies $zgz^{-1}=g$ for all $g\in G$. Since $\psi$ is a bijection, $\psi(z)$ will satifies $\psi(z)h\psi(z)^{-1}=h$ for all $h\in H$ and hence $\psi(z)$ lies in the center. This proves that $\psi (Z(G)) \subset Z(H)$. Now the inverse map $\psi^{-1}$ is also a quandle isomorphism and hence $\psi^{-1}(Z(H))\subset Z(G)$. This completes the proof of $1)$. Since $\psi$ is a quandle morphism $c_{\psi(g)}\circ \psi=\psi \circ c_g$. Hence, $c_{\psi(g)}= \psi \circ c_g \circ \psi^{-1}$. It follows that we have a well defined map $Inn(\psi) : Inn(G) \to Inn(H)$ given by $Inn(\psi)(c_g)=c_{\psi(g)}$. It follows also from the last identity that $Inn(\psi)$ is a morphism and that this morphism is invertible. This proves $2)$.
\end{proof}
\begin{corollary}
If $G$ is centerless and $Conj(G)$ is isomorphic to $Conj(H)$ for some group $H$, then $H$ is centerless and $G$ and $H$ are isomorphic and hence isoclinic.
\end{corollary}
\begin{proposition}
If two groups $G$ and $H$ have isomorphic conjugacy quandles and both centers intersect the derived groups trivially then $G$ and $H$ are isoclinic.
\end{proposition}
\begin{proof}
By the previous propositon $Inn(G)$ and $Inn(H)$ are isomorphic. We hence have a commutative diagram :

$$\begin{tikzcd}
Inn(G)\times Inn(G)\arrow{r}{\alpha}\arrow{d}{[\:\:,\:\:]_{Inn(G)}} &Inn(H)\times Inn(H) \arrow{d}{[\:\:,\:\:]_{Inn(H)}} \\ 
Inn(G)' \arrow{r}{Inn(\psi)}  & Inn(H)' 
\end{tikzcd} $$
where $\alpha= Inn(\psi)\times Inn(\psi)$ and $Inn(\psi)$ is the isomorphism of the previous proposition associated to a the quandle isomorphism $\psi :Conj(G)\to Conj(H)$. $Inn(G)$ and $Inn(H)$ are isomorphic to $G/Z(G)$ and $H/Z(H)$ respectively and the condition in the proposition implies that the derived groups of $Inn(G)$ and $Inn(H)$ are isomorphic respectevely to $G'$ and $H'$. The proposition follows.
\end{proof}

\begin{lemma}
Let $f: G \to H$ be a surjective group morphism. If $f^{-1}(Z(H))=Z(G)$ and $f$ restricts to an isomorphism from $G'$ to $H'$, then $G$ and $H$ are isoclinic.
\end{lemma}
\begin{proof}
The conditions $f^{-1}(Z(H))=Z(G)$ and $f$ surjective implie that $f$ induces an isomorphism $\bar{f}:G/Z(G)\to H/Z(H)$. Now take $g$ and $g'$ in $G$. We have that :
$$f([\bar{g},\bar{g}']_G)=[f(g),f(g')]_H=[\overline{f(g)},\overline{f(g')}]_H=[\bar{f}(\bar{g}),\bar{f}(\bar{g'})]_H.$$
This proves that the isomorphisms $\bar{f}: G/Z(G) \to H/Z(H)$ and the restriction of $f$ to $G'\to H'$ are compatible and hence $G$ and $H$ are isoclinic.
\end{proof}
For $Q$ a quandle one associates a group $G(Q)$ called the enveloping group :
$$G(Q)=\langle e_x ,\:x\in Q \vert e_xe_ye_x^{-1}=e_{x\triangleright y} \: \text{for} \: x,y \in Q \rangle.$$
The enveloping group is also known as associated group, structure group or adjoint group. The map $\varphi_Q : Q \to G(Q)$ mapping $x\in Q$ to $e_x$ is universal with respect to quandle morphisms $Q\to Conj(G)$; meaning if $f:Q \to Conj(G)$ is a quandle morphism then there exists a unique group morphism $f_{G(Q)}:G(Q)\to G$ such that $f_{G(Q)} \circ \varphi_Q=f$.\\\\
For $G$ a group, we will \textbf{denote} by $A(G)$ the enveloping group of $Conj(G)$ and by $H_S^2(G,\mathbb{C}^\times)$ the subgroup of the Schur multiplier $H^2(G,\mathbb{C}^\times)$ consisting of classes of symmetric $2$-cocycles, i.e. cocycles $\alpha :G\times G \to \mathbb{C}^\times$ such that $\alpha(g,h)=\alpha(h,g)$ for all $h,g \in G$.
\begin{proposition}
Let $G$ be a finite group. If $H^2_S(G,\mathbb{C}^\times)=0$, then the enveloping group $A(G)$ of $Conj(G)$ and $G$ are isoclinic.
\end{proposition}
\begin{proof}
The identity $id_G:G\to G$ is a quandle morphism. Hence by the universal property of $\varphi_{Conj(G)}:G \to A(G)$ we have a unique group morphism $\pi :A(G)\to G$ such that $\pi \circ \varphi_{Conj(G)}=id_G$. The mprphism $\pi$ is onto since it has a right inverse. Since it is a surjective group morphism $\pi(Z(A(G)))\subset Z(G)$. Now for $z\in Z(G)$ it follows from the definition of $A(G)$ that $e_z$ is in the center of $A(G)$ and $\pi(e_z)=\pi \circ \varphi_{Conj(G)}(z)=z$. This proves that $\pi(Z(A(G)))= Z(G)$. Moreover, the kernel of $\pi$ lies in the center of $A(G)$ (\cite{MM}). Hence, $\pi^{-1}(Z(G))=Z(A(G))$. The morphism $\pi$ restict to an isomorphism of derived groups $A(G)'\to G'$ if $H^2_S(G,\mathbb{C}^\times)=0$ (\cite{MM}). The last two condition proven for $\pi$ implie by the previous lemma that $A(G)$ and $G$ are isoclinic. We have proved the proposition.
\end{proof}
\begin{proposition}
If $G$ and $H$ are two groups having isomorphic conjugacy quandles and satisfying $H^2_S(G,\mathbb{C}^\times)=H^2_S(H,\mathbb{C}^\times)=0$, then $G$ and $H$ are isoclinic.
\end{proposition}
\begin{proof}
If $G$ and $H$ have isomorphic conjugacy quandles then $A(G)$ and $A(H)$ are isomorphic. It follows from the previous propostion since isoclinism is an equivalence relation that $G$ and $H$ are isoclinic.
\end{proof}
The Bogomolov multiplier $B_0(G)$ of a group $G$ is the subgroup of $H^2(G,\mathbb{Q}/\mathbb{Z})$ consisting of classes whose restrictions to any abelian subgroup of $G$ are trivial. 
\begin{proposition}\cite{MM}
The group $H_S^2(G,\mathbb{C}^\times)$ is isomorphic to a subgroup of $B_0(G)$.
\end{proposition}
\begin{proof}
We review the idea of the proof from \cite{MM}. Let $B_\mathbb{C}(G)$ be the subgroup of $H^2(G,\mathbb{C}^\times)$ consisting of classes whoset restriction to any abelian subgroup of $G$ is trivial. The group $H^2_S(G,\mathbb{C}^\times)$ lies in $B_\mathbb{C}(G)$ and the last group is isomorphic to $B_0(G)$.
\end{proof}
We get by combining the last two propositions that :
\begin{proposition}\label{prop}
Two groups with trivial Bogomolov multipliers and having isomorphic conjugacy quandles are isoclinic. 
\end{proposition}
\section{Results using GAP}
\begin{proposition}\label{prop1}
\item[1)] A group of order $n<128$ has trivial Bogomolov multiplier if $n\neq 64$.
\item[2)] There are exactly 9 isomorphism classes of groups of order $64$ with non trivial Bogolomov multiplier. The groups of these classes are isoclinic.
\end{proposition}
\begin{proof}
For $n$ an integer the command $AllSmallGroups(n)$ in GAP returns a list of groups of order $n$. This list contains no two isomorphic groups and runs over all isomorphism classes of groups of the corresponding order. The BogoMolov multiplier of a group $G$ can be computed using the command $BogomolovMultiplier(G)$ from the "hap" package. The following code defines a function named "Bog" that takes an integer as an argument and that runs over all the lists of $AllSmallGroups(i)$ for $i<n$ and returns a table of groups of these lists with non trivial Bogomolov multiplier :\\\\
LoadPackage("hap");\\
Bog:=function(n)\\
local gr,m,i, j,tab;\\
tab:=[\:];\\
for i in [1..n] do\\ 
gr:=AllSmallGroups(i);\\
m:=Length(gr);\\
for j in [1..m] do\\
if BogomolovMultiplier(gr[j])<>[\:] then\\
Add(tab, gr[j]);\\
fi;\\
od;\\
od;\\
return tab;\\
end;\\\\
We add to the above code the lines :\\\\
A:=Bog(127);\\
a:=Length(A);\\
B:=List(A,x->Size(x));\\\\
So $A$ in the code is the list of groups from $AllSmallGroups$ of order $n<128$ and having non trivial Bogomolov multiplier. The command $a:=Length(A)$ returns $9$ and hence the list contains nine elements. Proving that there are exactly $9$ isomorphism classes of groups of order $n<128$ with non trivial multiplier. $B:=List(A,x->Size(x));$ returns the list of orders of the groups in $A$. The result is a table filled with $64$. Hence, the groups with order $n<128$ and with non trivial Bogolomov multiplier are all of order $64$. Now the command $IsoclinismClasses$ takes a list of groups and outputs a list of equivalence classes under isoclinism. Adding to gap the line codes :\\\\
C:=IsoclinismClasses(A);\\\\
c:=Length(C);\\\\\
we see that $c$ is equal to $1$ and hence $A$ constitute one equivalence class under isoclinism. This completes the proof of the proposition
\end{proof}
Combining the above proposition with proposition \ref{prop}, we get :
\begin{proposition}\label{prop2}
Two groups with order $n<128, n\neq 64$ have isomorphic conjugacy quandles if and only if they are isoclinic.
\end{proposition}
Since the groups of order $64$ with non trivial Bogolomov multiplier are all isoclinc it follows from the theorem of the first section that :
\begin{proposition}
Groups of order $64$ with non trivial Bogomolov multiplier have isomorphic conjugacy quandle
\end{proposition}
\begin{proposition}
Let $G$ be a group of order $64$ with non trivial Bogolomov multiplier. A group $H$ of order $64$ have a conjugacy quandle isomorphic to the one of $G$ if and only if $H$ has a non trivial Bogomolov multiplier. 
\end{proposition}
\begin{proof}
We first implement a function "Env" that takes as an argument a finite group $G$ a returns the presented group $A(G)$, the enveloping group of $Conj(G)$ :\\\\
Env := function(G)\\
local elms, r, i, j, h, po, k, F, gens, rels, H; \\
elms := Elements(G);\\
r := Length(elms);\\
F := FreeGroup(r); \\
gens := GeneratorsOfGroup(F);\\
rels := [\:];\\
for i in [1..r] do\\
for j in [1..r] do\\
        h := elms[j];\\
        po:=elms[i]*elms[j]*elms[i] $\widehat{\:\:}$-1;\\
        k := Position(elms, po);\\
        Add(rels, gens[i] * gens[j] * gens[i] $\widehat{\:\:}$-1*gens[k] $\widehat{\:\:}$-1);\\
od;\\
od;\\
H := F / rels;\\
return H;\\
end;\\\\
For a group $G$ we will denote by $\Gamma_3^2G$ the third term of the lower exponent-$2$ central series. One can compute the quotient $G/\Gamma_3^2(G)$ using GAP. Recall that the list $A$ of the proof of the previous proposition is the list of groups of $AllSmallGroups(64)$ with non trivial Bogolomov multiplier. The following code returns the list $tab$ of groups $G$ of $AllSmallGroups(64)$ such that $\vert A(G)/\Gamma_3^2A(G)\vert=\vert A(G_0)/\Gamma_3^2A(G_0)\vert$ where $G_0$ is the group corresponding to the first group of the list $A$.\\\\
gr:=AllSmallGroups(64);\\
m:=Length(gr);\\
tab:=[\:];\\
G:=Env(A[1]);\\
G:=EpimorphismPGroup(G,2,2);\\
G:=Image(G);\\
for x in gr do\\
H:=Env(x);\\
H:=EpimorphismPGroup(H,2,2);\\
H:=Image(H);\\
if Size(H)=Size(G) then\\
Add(tab, x);\\
fi;\\
od;\\\\
The above code could take $2$ minutes to finish the computations. Now adding the code line :\\\\
Length(tab);\\\\
We find the list $tab$ contains $9$ groups. Since the $9$ groups of $A$ are isoclinc, they have isomorphic conjugacy quandle and hence isomorphic enveloping groups. Hence the lists $tab$ and $A$ coincide. This proves that if $Conj(G)$ is isomorphic to $Conj(H)$ and $G$ is of order $64$ with non trivial Bogmolov multiplier and $H$ is of order $64$ then $H$ has a non trivial Bogolomov multiplier.
\end{proof}
\begin{proposition}
Let $G$ and $H$ be two groups of order $64$. $G$ and $H$ have isomorphic conjugacy quandles if and only if $G$ and $H$ are isoclinic.
\end{proposition}
\begin{proof}
Let $G$ and $H$ be as in the proposition. If $G$ and $H$ has trivial Bogomolov multiplier then they have isomorphic conjugacy quandles if and only if they are isoclinic (by proposition \ref{prop} and the theorem of the first section). Say $G$ has a non trivial Bogolomov multiplier and $G$ and $H$ have isomorphic conjugacy quandles. By the previous proposition the Bogolomov multiplier of $H$ is non trivial. Hence, by proposition \ref{prop1} $H$ and $G$ are isoclinic. This with the theorem of the first section completes the proof of the proposition.
\end{proof}
The last proposition with proposition \ref{prop2} give :
\begin{theorem}
Let $G$ and $H$ be two groups of order $n<128$. $G$ and $H$ has isomorphic conjugacy quandles if and only if they are isoclinic.
\end{theorem}
We will now construct two groups of order $128$ having isomorphic conjugacy quandles but are not isoclinic. We will \textbf{denote} by $G_1$ the group of order $64$ corresponding to $SmallGroup(64,149)$ of GAP. Recall that we have called a $2$-cocycle $\alpha$ with values in $\mathbb{C}^\times$ symmetric if $\alpha(g,h)=\alpha(h,g)$ for all $h,g$ in the group. 
\begin{proposition}\cite{MMb}\label{prop3}
The group $G_1$ admits a symmetric $2$-cocycle $\alpha :G_1\times G_1 \to \mathbb{C}^\times$ having a non trivial cohomology class in $H^2(G,\mathbb{C}^\times)$. 
\end{proposition}
\begin{proposition}
The Bogomolov multiplier of $G_1$ is isomorphic to $\mathbb{Z}/2\mathbb{Z}$.
\end{proposition}
\begin{proof}
We use the "hap" pacakage of GAP and the commands :\\\\ G1:=SmallGroup(64,149); \\BogomolovMultiplier(G1);
\end{proof}
\begin{proposition}
the group $G_1$ admits a normalised symmetric $2$-cocycle $\beta :G_1\times G_1 \to \mathbb{C}^\times$ having a non trivial cohomology class in $H^2(G,\mathbb{C}^\times)$ and taking values in $\{1,-1\}$.
\end{proposition}
\begin{proof}
Take the cocycle $\alpha$ of proposition \ref{prop3}. Since $H_S^2(G,\mathbb{C}^\times)$ is isomorphic to a subgroup of $B_0(G_1)$ and $B_0(G_1)$ is the group on two elements, the class of $\alpha$ is of order $2$ and $\alpha^2$ is a coboundary, i.e. there is a map $\lambda :G\to \mathbb{C}^\times$ such that :
$$\alpha^2(g,h)=\lambda(g)\lambda(h)\lambda(gh)^{-1}.$$
The coboundary $\lambda(g)\lambda(h)\lambda(gh)^{-1}$ is hence symmetric and its square root $\sqrt{\lambda(g)}\sqrt{\lambda(h)}\sqrt{\lambda(gh)^{-1}}$ is a symmetric coboundary. The $2$-cocycle $\beta_0$ given by :
$$\beta_0(g,h)=\alpha(g,h) \sqrt{\lambda(g)}^{-1}\sqrt{\lambda(h)}^{-1}\sqrt{\lambda(gh)},$$
is hence a symmetric cocycle cohomologus to $\alpha$ taking values in $\{1,-1\}$ indeed its square is the constant function equal to $1$. Now take $\mu:G \to \mathbb{C}^\times$ given by $\mu(1)=\beta_0(1,1)^{-1}$ and $\mu(g)=1$ for $g\in G\setminus \{1\}$. The desired cocycle of the proposition can be taken to be equal to the cocycle $\beta$ given by :
$$\beta(g,h)=\beta_0(g,h)\mu(g)\mu(h)\mu(gh)^{-1}.$$
\end{proof}
Let $\beta$ be as in the previous proposition and let $\mathbb{C}^\times \times_\beta G_1$ be the group whose underlying set is $\mathbb{C}^\times \times G_1$ and whose product is given by:
$$(x,g)(y,h)=(xy\beta(g,h),gh).$$
We denote by $\pi$ the natural projection $\mathbb{C}^\times \times_\beta G_1 \to G_1$. This projection is a group morphism and we have a central extension : 
$$1\to \mathbb{C}^\times\times 1 \to \mathbb{C}^\times \times_\beta G_1 \overset{\pi}{\to} G_1\to 1.$$
\begin{proposition}
\item[1)] The derived group of  $\mathbb{C}^\times \times_\beta G_1$ intersects $\mathbb{C}^\times \times 1$ non trivially. 
\item[2)] The derived group of $\mathbb{C}^\times \times_\beta G_1$ contains more elements then the derived group of $G_1$.
\end{proposition}
\begin{proof}
Asume that the derived group does intersect $\mathbb{C}^\times \times 1$ trivially. Denote by $f:  \mathbb{C}^\times \times_\beta G \to  (\mathbb{C}^\times \times_\beta G_1 )^{ab}$ the abelianisation morphism. By the assumption we made, $f$ restricted to $\mathbb{C}^\times \times 1$ is injective. Since $\mathbb{C}^\times$ is divisible, the abelianisation of $ \mathbb{C}^\times \times_\beta G_1 $ splits as :
$$( \mathbb{C}^\times \times_\beta G_1 )^{ab}=f(\mathbb{C}^\times \times 1 )\oplus A,$$
for some abelian group $A$. Denote by $H$ the subgroup of $ \mathbb{C}^\times \times_\beta G_1 $ equal to $f^{-1}(A)$. One can check that $(\mathbb{C}^\times \times 1)\cap H=1$ and $(\mathbb{C}^\times \times 1)H= \mathbb{C}^\times \times_\beta G_1 $. This implies that the restriction of $\pi$ to $H$ is a group isomorphism and that the sequence splits wich contradicts the fact that the $2$-cocycle $\beta$ has a non trivial cohomology class. This proves $1)$. We prove $2)$, $\pi$ restricts to an onto morphism from the derived group of $\mathbb{C}^\times \times_\beta G_1$ to the derived group of $G_1$. By $1)$ of this proposition the kernel of this morphism is non trivial. This proves $2)$. 
\end{proof}
From now on we will \textbf{denote} by $G$ the subgroup of $\mathbb{C}^\times \times_\beta G_1$ whose elements are the elements of the subset $\{1,-1\}\times G_1 \subset \mathbb{C}^\times \times_\beta G_1$. This is indeed a subgroup since $\beta$ take values in $\{1,-1\}$. The group $G$ has order $128$.
\begin{proposition}
The derived subgroup $G'$ of $G$ has more elements then $G_1'$.
\end{proposition}
 \begin{proof}
Since the cocycle is normalized and the elements of the form $(x,1)$ are central we have that for $(x,g),(y,h)\in \mathbb{C}^\times \times_\beta G_1$ :
$$[(x,g),(y,h)]=[(x,1)(1,g),(y,1)(1,h)]=[(1,g),(1,h)].$$
Hence the derived subgroup of $\mathbb{C}^\times \times_\beta G_1$ is a subgroup of the derived subgroup of $G$ and hence they are both equal. This proposition is therefore a consequence of statement $2)$ the previous one.
\end{proof}
\begin{lemma}\cite{MM}\label{lem}
A symmetric $2$-cocycle $\alpha: G\times G\to \mathbb{C}^\times$ satisfies the eqaution :
$$\alpha(g,hg^{-1})\alpha(h,g^{-1})\alpha(1,1)^{-1}\alpha(g,g^{-1})^{-1}=1.$$
\end{lemma}
\begin{proof}
By replacing $h$ by $hg$ in the equation, we get the "equivalent" equation
\begin{equation}\label{eq2}
\alpha(g,h)\alpha(hg,g^{-1})\alpha(1,1)^{-1}\alpha(g,g^{-1})^{-1}=1.
\end{equation} 
We will prove that the last equation holds. By the cocycle condition : 
$$\alpha(hg,g^{-1}) \alpha(h,g)\alpha(h,1)^{-1}\alpha(g,g^{-1})^{-1}=1.$$
Hence, the right hand side of equation (\ref{eq2}) is equal to :
$$\alpha(g,h)\alpha(1,1)^{-1}\alpha(h,g)^{-1}\alpha(h,1).$$
Using the cocycle condition one proves that $\alpha(1,1)=\alpha(h,1)$ and by symmetry of the cocycle $\alpha(g,h)=\alpha(h,g)$. Hence, the right-hand side of equation (\ref{eq2}) is equal to $1$ and equation (\ref{eq2}) holds. This proves the lemma.
 \end{proof}
Let $H$ be the group (direct product) $\{1,-1\}\times G_1$. $H$ has order $128$.
\begin{proposition} 
$G$ and $H$ has isomorphic conjugacy quandles but are not isoclininc.
\end{proposition}
\begin{proof}
The derived subgroup of $H$ is isomorphic to $G_1'$ and the derived subgroup of $G$ has more elements then $G_1'$. Hence, $G$ and $H$ can't be isoclinic. We will show that they have isomorphic conjugacy quandles. $G$ and $H$ has the same underlying sets. Denote by $\psi: H\to G$ the identity mapping. For $(x,g),(y,h)\in H$, we have that 
$$\psi((x,g)(y,h))=\beta(g,h)^{-1}\psi(x,g)\psi(y,h).$$
where $\beta(g,h)$ is identified to $(\beta(g,h),1)$. Hence, 
\begin{align*}
\psi((x,g)(y,h)(x,g)^{-1})&=\beta(g,hg^{-1})^{-1}\psi(x,g)\psi((y,h)(x,g)^{-1})\\
&=\beta(g,hg^{-1})^{-1}\beta(h,g^{-1})^{-1}\psi(x,g)\psi(y,h)\psi((x,g)^{-1})\\
&=\beta(g,hg^{-1})^{-1}\beta(h,g^{-1})^{-1}\beta(g,g^{-1})\psi(x,g)\psi(y,h)\psi(x,g)^{-1}\\
&=\psi(x,g)\psi(y,h)\psi(x,g)^{-1},
\end{align*}
where the last equality follows from the last lemma since $\beta$ is normalized and symmetric. This equation proves that the bijection $\psi$ is a quandle morphism and hence $G$ and $H$ has isomorphic conjugacy quandles.
\end{proof}

 \end{document}